\documentclass[12pt]{amsart}
\usepackage{amssymb}
\usepackage{extarrows}
\makeatletter
\@namedef{subjclassname@2010}{%
	\textup{2010} Mathematics Subject Classification}
\makeatother
\usepackage[T1]{fontenc}
\newtheorem{theorem}{Theorem}[section]
\newtheorem{corollary}[theorem]{Corollary}
\newtheorem{lemma}[theorem]{Lemma}
\newtheorem{proposition}[theorem]{Proposition}

\theoremstyle{definition}
\newtheorem{definition}[theorem]{Definition}
\newtheorem{remark}[theorem]{Remark}
\newtheorem{example}[theorem]{Example}

\numberwithin{equation}{section}
\begin{document}
	\baselineskip=17pt
	\title{ Continuous orbit equivalence of semigroup actions  }

	\author[X. Q. Qiang]{Xiangqi Qiang}
	\address{School of Mathematical Science \\ Yangzhou University\\
		Yangzhou 225002, China}
	\email{905163754@qq.com}
	
	\author[C. J. Hou]{Chengjun Hou}
	\address{School of Mathematical Science \\ Yangzhou University\\
		Yangzhou 225002, China}
	\email{cjhou@yzu.edu.cn}

	\begin{abstract}
		In this paper, we consider semigroup actions of discrete countable semigroups on compact spaces by surjective local homeomorphisms. We introduce notions of continuous one-sided orbit equivalence and continuous orbit equivalence  for semigroup actions, and characterize them in terms of the   corresponding  semi-groupoids and transformation groupoids respectively.  Finally, we consider   the case of semigroup  actions by homeomorphisms and    relate   continuous orbit  equivalence of semigroup actions to that of group actions.

	\end{abstract}
	
	
	
	\maketitle

	\section{Introduction }
	
Inspired by ergodic theory, Giordano, Putnam and Skau introduced in \cite{GPS} the topological version of orbit equivalences. They obtained a breakthrough result that two Cantor minimal homeomorphisms are strongly orbit equivalent if and only if the crossed product $C^*$-algebras associated with two systems are isomorphic. In \cite{BT}, Boyle and Tomiyama  characterized continuous orbit equivalence between topologically free homeomorphisms. Lin and  Matui gave a few complete descriptions for relations  between the proposed approximate versions of conjugacy and the corresponding crossed product $C^*$-algebras for Cantor minimal systems via $K$-theory.  Especially, they showed that the approximate $K$-conjugacy is the same as strong orbit equivalence for Cantor minimal systems (\cite{LM}) and these systems are (topologically) orbit equivalent if and only if the associated crossed products are tracially equivalent (\cite{Lin}).   In \cite{Ma1}, Matsumoto introduced the notion of  continuous orbit  equivalence for one-sided topological Markov shifts, which are local homeomorphisms, and showed that two irreducible one-sided topological Markov shifts are continuously orbit  equivalent if and only if there exists a diagonal preserving $C^*$-isomorphism between the associated Cuntz-Krieger algebras.  Using the groupoid technique, Matsumoto  and Matui showed in \cite{MM} that this is equivalent to the existence of an isomorphism of two canonical groupoids associated to one-sided shifts. These results were in \cite{CEOR} generalized from the reducible to the general case.

Recently, the concept of continuous orbit equivalence has been generalized to many different cases. In \cite{Li1,Li2},  Li introduced the notions of continuous orbit equivalence for continuous group actions and partial group actions, and  characterized them in terms of  isomorphisms of (partial) $C^*$-crossed products preserving  Cartan subalgebras. Later,  Cordeiro  and  Beuter  extended in \cite{CB} Li's results to partial actions of inverse semigroups and characterized orbit equivalence of topologically principal systems. In \cite{HQ,QH}, motivated by  Mastumoto's notion of asymptotic continuous orbit equivalence in Smale spaces (\cite{Ma2}), we  characterized continuous orbit equivalence of expansive systems up to local conjugacy relations and classified automorphism systems of \'{e}tale equivalence relations up to continuous orbit equivalence. For more interesting progress  and applications on continuous orbit equivalence, see \cite{BCW,CRST,Ma3} and the references therein.
	
 For a semigroup action $(X,P,\theta)$ of a countable semigroup $P$ on a compact space $X$ by surjective local homeomorphisms, Exel and Renault extended this action to an \emph{interaction group} and defined a \emph{transformation groupoid} whose $C^*$-algebra turns to be isomorphic to the crossed product for the interaction group under some standing hypotheses (\cite{ER}). The aim of this paper is to develop the relationship among operator algebras, transformation groupoids and semigroup actions.

 Given a semigroup action $(X,P,\theta)$,  the sets $[x]_{\theta,s}	=\{ \theta_{m}(x): m\in P\}$  and
	$[x]_{\theta}=\{y\in X: \theta_{m}(x)=\theta_{n}(y) \hbox{ for } m,n\in P\}$ are the one-sided orbit  and	the full orbit of $x$, respectively. As in the group actions, two semigroup actions are said to be \emph{one-sided orbit equivalent (resp. orbit equivalent)} if there is a homeomorphism preserving corresponding orbits between underling compact  spaces. Similarly, we can consider continuous versions of these two orbit equivalence. We say that semigroup actions $(X,P,\theta)$ and $(Y,S,\rho)$ are \emph{continuously one-sided orbit equivalent} if  there exist a homeomorphism $\varphi$ from $X$ onto $Y$ and continuous maps
	$a: P\times X \rightarrow S$ and $b: S\times Y \rightarrow P $ such that $\varphi(\theta_{m}(x))=\rho_{a(m,x)} (\varphi(x))$ and $ \varphi^{-1}(\rho_{s}(y))=\theta_{b(s,y)}(\varphi^{-1}(y))$ for all
$m \in P$, $x\in X$, $s\in S$ and $y\in Y$. They are called to be \emph{continuously orbit equivalent} if  there exist a homeomorphism $\varphi:\, X\rightarrow Y$,  continuous mappings
	$a_{1} ,b_{1} :\;  \cup_{(m,n)\in P\times P}(\{(m,n)\}\times X_{(m,n)})\rightarrow S$    and  $a_{2} ,b_{2}: \cup_{(s,t)\in S\times S}(\{(s,t)\}\times Y_{(s,t)}) \rightarrow P $ such that
$\rho_{a_{1}(m,n,x,y)}(\varphi(x))=\rho_{b_{1} (m,n,x,y)}(\varphi(y))$ and $\theta_{a_{2} (s,t,u,v)}(\varphi^{-1}(u))= \theta_{b_{2}(s,t,u,v)}(\varphi^{-1}( v))$ for $ (x,y)\in X_{(m,n)}$, $(u,v)\in Y_{(s,t)}$, where  $X_{(m,n)}=\{(x,y)\in X\times X:\, \theta_{m}(x)=\theta_{n}(y) \}$ and $Y_{(s,t)}=\{(u,v)\in Y\times Y:\, \rho_{s}(u)=\rho_{t}(v)\}$. Denote by $P\ltimes X$ and $\mathcal{G}(X,P,\theta)$  the semi-groupoid and transformation groupoid associated to $(X,P,\theta)$, respectively.  In particular, if $(X,P,\theta)$ is a semigroup  action by homeomorphisms and $G$ is a countable group   containing $P$ as a (unital) sub-semigroup and $G=P^{-1}P=PP^{-1}$, then we can extend  $(X,P,\theta)$ to be a group action $(X,G,\widetilde{\theta})$.  The followings are main results in this paper.
	
\begin{theorem} Let $(X,P,\theta)$ and $(Y,S,\rho)$ be two essentially free semigroup actions. Then
 \begin{enumerate}
 \item[(i)] $(X,P,\theta)$ and $(Y,S,\rho)$ are continuously one-sided orbit equivalent if and only if semi-groupoids $P\ltimes X$ and $S\ltimes Y$ are (topologically) isomorphic.
\item[(ii)] If $(X,P,\theta)$ and $(Y,S,\rho)$ are continuously orbit equivalent, then two \'{e}tale groupoids  $\mathcal{G}(X,P,\theta)$ and $\mathcal{G}(Y,S,\rho)$ are  (topologically) isomorphic.
 \item[(iii)] If $X$ and $Y$ are totally disconnected, then $(X,P,\theta)$ and $(Y,S,\rho)$ are continuously orbit equivalent if and only if $\mathcal{G}(X,P,\theta)$ and $\mathcal{G}(Y,S,\rho)$ are  (topologically) isomorphic if and only if there is a $*$-isomorphism $\Phi$ from $C_r^*(\mathcal{G}(X,P,\theta))$ onto $C_r^*(\mathcal{G}(Y,S,\rho))$ such that $\Phi(C(X))=C(Y)$.
 \item[(iv)] Assume that $(X,P,\theta)$ and  $(Y,S,\rho)$ are semigroup actions by homeomorphisms. If
		$(X,P,\theta)$ and $(Y,S,\rho)$ are continuously orbit equivalent, then the associated group actions $ (X,G,\tilde{\theta})$ and $(Y,H,\tilde{\rho})$ are continuously orbit equivalent in Li's sense $($\cite{Li1}$)$. Moreover, if $X$ and $Y$ are totally disconnected or $(G,P)$ and $(H,T)$ are two lattice-ordered groups, then the converse of this statement holds.
\end{enumerate}
\end{theorem}
Here the notion of essential  freeness for $(X,P,\theta)$   is derived from the study of the groupoid associated to the one-sided shift. It's worth noting that condition (I) in   one-sided topological Markov shift guarantees that this system is essentially free. Thus the above result generalizes Matsumoto and Carlsen et. al.'s results for one-sided subshifts of finite type (\cite{MM,CEOR}).

 We now give some notions needed in this paper.  For a topological  groupoid $\mathcal{G}$, let $\mathcal{G}^{(0)}$ and $\mathcal{G}^{(2)}$ be  the unit space and the set of composable pairs, respectively. The range map $r$ and the domain map $d$ from $\mathcal{G}$ onto $\mathcal{G}^{(0)}$ are defined by $r(g)=gg^{-1}$ and $d(g)=g^{-1}g$, respectively. A subset $U$ of  groupoid $\mathcal{G}$ is a bisection if both the restrictions of $r$ and $d$ to $U$ are injective.   If $r$ and $d$ are local homeomorphisms, then $\mathcal{G}$ is called to be \'{e}tale.
We refer to \cite{Re1,Sim} for more details on topological groupoids and their $C^*$-algebras.

This paper is organized as follows. In section 2, we list a number of terminologies used in the paper and characterize continuous one-sided orbit equivalence for semigroup actions by the associated semi-groupoids. In section 3,  we introduce the notion of continuous   orbit equivalence of semigroup actions and characterize it in terms of  the associated transformation groupoids, as well as   their reduced groupoid $C^{*}$-algebras with canonical Cartan subalgebras. In section 4,  we consider   the case of semigroup  actions by homeomorphisms and  discuss the relationship between   continuous orbit  equivalence of semigroup actions and that of group actions.

	\section{ Semigroup actions and one-sided orbit equivalence}
	
	Let  $X$  be  a second-countable compact Hausdorff space, $G$  a countable discrete group and $P$ a subsemigroup of $G$. We assume that $X$ has no isolated points and $P$ contains the identity element $e$ of $G$ such that $G=P^{-1}P=PP^{-1}$. Denote by $End(X)$ the semigroup of all  surjective local  homeomorphisms on $X$ under the composition operation.  By  a right action $\theta$ of $P$ on $X$ we mean that  it is  a mapping   $\theta:\;n\in P\rightarrow \theta_{n}\in End(X)$  satisfying that $\theta_{n}\theta_{m}=\theta_{mn}$ for every $n,m \in P$  and $\theta_e=id_{X}$, the identity map on $X$. We denote by a triple $(X,P,\theta)$ a semigroup action in order to emphasize the base space $X$ and the semigroup $P$. In particular, when $P=G$ and each $\theta_n$ is a homeomorphism on $X$, we have a (right) group action $(X,G,\theta)$.
	
	There are two canonical algebraic structures attached to an action $(X,P,\theta)$. One is the topological semi-groupoid, $P\ltimes X: =\{(m,x):\, m\in P, x\in X\}$, whose topology is the  product topology and  multiplication is as follows (\cite{Ex}):
$$ (m,x)(n,y)=(nm, y) \mbox{ if $x=\theta_n(y)$}.$$ The other is the transformation groupoid
	$$\mathcal{G}(X,P,\theta)\,:=\left\lbrace (x,g,y)\in X\times G\times X :\exists m,n \in P , g=mn^{-1},\theta_{m}(x)=\theta_{n}(y)\right\rbrace,$$
	which is a second-countable locally compact Hausdorff \'{e}tale groupoid  under the following multiplication and inverse,
	$$(x,g,y)(u,h,v)=(x,gh,v) \,\, \hbox{if}\,\, y=u,$$
	$$(x,g,y)^{-1}=(y,g^{-1},x),$$
	and the topology with basic open sets $$\Sigma(U,m,n,V)\,:=\left\lbrace (x,mn^{-1},y)\in \mathcal{G}(X,P,\theta):\;  \theta_{m}(x)=\theta_{n}(y) , x\in U,y\in V  \right\rbrace,$$ indexed by quadruples $(U,m,n,V)$, where $m,n\in P$, $U$ and $V  $ are open subsets of $X$,  $\theta_{m}|_{U} ,\theta_{n}|_{ V} $ are homeomorphisms, and $\theta_{m} (U) =\theta_{n} ( V)  $ (\cite{ER}).	
	
	If we identify the unit space $\mathcal{G}(X,P,\theta)^{(0)}=\left\lbrace (x,e,x):\;x\in X\right\rbrace $    with $X$ by identifying $(x,e,x)$ with $x$, then $r(x,g,y)=x$ and $d(x,g,y)=y$. One can check that the  mapping $c_{ \theta}:\mathcal{G}(X,P,\theta) \rightarrow G$ defined by $c_{ \theta}(x,g,y)=g$ is a continuous cocycle.	
	
	Given a semigroup action $(X,P,\theta)$, for $x\in X$, we call  sets 	$$[x]_{\theta,s}\, :=\{ \theta_{m}(x): m\in P\}$$ and
	$$[x]_{\theta}\, :=\{y\in X: \exists m,n\in P\; \text{such that } \theta_{m}(x)=\theta_{n}(y)  \}$$
	the one-sided orbit   and	the full orbit of $x$ under $\theta$, respectively.


	\begin{definition} Let $(X,P,\theta)$ and $(Y,S,\rho)$ be two semigroup  actions.
		\begin{enumerate}
			\item[(i)] We say they are \emph{conjugate} if there exist  a homeomorphism $\varphi: X\rightarrow Y$ and a semigroup isomorphism $\alpha: P\rightarrow S$ such that  $\varphi\theta_m=\rho_{\alpha(m)}\varphi$ for each $m\in P$.
			
			\item[(ii)]  We say they  are \emph{ one-sided orbit equivalent} if there exists a homeomorphism $\varphi:\, X\rightarrow Y $ such that $\varphi([x ]_{\theta,s}) =[\varphi(x) ]_{\rho,s}$ for $x\in X$.
			
			\item[(iii)]	  	We say they are \emph{  orbit equivalent} if there exists a homeomorphism $\varphi:\, X\rightarrow Y $ such that $\varphi([x ]_{\theta}) =[\varphi(x) ]_{\rho}$ for $x\in X$.
			
		\end{enumerate}
	\end{definition}
	
	
	Clearly, conjugacy between two semigroup actions implies one-sided orbit equivalence and orbit equivalence in turn. In addition, if $(X,P,\theta)$ and $(Y,S,\rho)$ are \emph{ one-sided orbit equivalent} via a homeomorphism $\varphi$, then for each $m\in P$ and $x\in X$, there exists $a(m,x)$ (depending on $m$ and $x$) in $S$ such that $\varphi(\theta_m(x))=\rho_{a(m,x)}(\varphi(x)).$ Symmetrically, for each $s\in S$ and $y\in Y$, there exists $b(s,y)$ (depending on $s$ and $y$) in $P$ such that $\varphi^{-1}(\rho_s(y))=\theta_{b(s,y)}(\varphi^{-1}(y))$.
	Thus we have following continuous version of one-sided orbit equivalence  which is analogous to \cite{Li1}.
	
	
	\begin{definition} We say two semigroup actions $(X,P,\theta)$ and $(Y,S,\rho)$
		are \emph{continuously one-sided orbit equivalent}  ( we write $(X,P,\theta)\sim_{ csoe}(Y,S,\rho)$) if  there exist a homeomorphism $\varphi:\, X\rightarrow Y$,  continuous mappings
		$a :\; P\times   X \rightarrow S$    and  $ b  :\; S\times   Y \rightarrow P $ such that  $$\varphi(\theta_{m}(x))=\rho_{a(m,x)} (\varphi(x))\;\; \text{for} \;  m \in P,x\in X \eqno{(2.1)}$$
		$$ \varphi^{-1}(\rho_{s}(y))=\theta_{b(s,y)}(\varphi^{-1}(y))\;\; \text{for} \;  s  \in S, y\in Y. \eqno{(2.2)}$$
		
		
	\end{definition}
	
	In the rest of this section, we will characterize   continuous  one-sided orbit equivalence of semigroup actions    in terms of  the associated  semi-groupoids. The following definition comes from \cite{Re2}.
	
	\begin{definition}
		A	semigroup  action  $(X,P,\theta)$ is said to be \emph{essentially free} if the interior of $\{x\in X:\;\theta_{m}(x)=\theta_{n}(x)\}$ in $X$ is empty for all  distinct pairs $ m,n\in P$.
	\end{definition}
	
	\begin{remark}
		If  $(Y,S,\rho)$ (resp. $(X,P,\theta)$) is essentially free,  then the map $a$ (resp. $b$) is uniquely determined by (2.1) (resp. (2.2)). In fact,   if $a' :\; P\times   X \rightarrow S$  is another continuous map such that  $\varphi(\theta_{m}(x))=\rho_{a'(m,x)} (\varphi(x))$ for $ m \in P, x\in X$, then from the continuity of $a$ and $a'$, for arbitrary $ m \in P,x\in X  $, there exists an open neighbourhood $U$ of $x$ such that $a$ and $a'$ are constant on $\{m\}\times U$ with values $a(m,x)$ and $a'(m,x)$. Thus for every $z\in U$,
		$\rho_{a(m,x)} (\varphi(z))=\varphi(\theta_{m}(z))=\rho_{a'(m,x)} (\varphi(z))$. Essential freeness of $(Y,S,\rho)$ implies $a(m,x)=a'(m,x)$.
	\end{remark}

	\begin{lemma}
		In the situation of Definition 2.2, assume that $(X,P,\theta)$ and $(Y,S,\rho)$ are essentially free. Then  $$a(nm,x)=a(n,x)a(m,\theta_{n}(x)) \mbox{ and } \,\,b(st,y)=b(s,y)b(t,\rho_{s}(y))$$ for  $n,m\in P$, $x\in X$ and $s,t\in S$, $y\in Y$.
	\end{lemma}
	\begin{proof}
		Let $n,m\in P$, $x\in X$ be arbitrary. Choose an open neighbourhood $U$ of $x$ such that $a(nm,x')=a(nm,x)$, $a(m,\theta_{n}(x'))=a(m,\theta_{n}(x))$ and $a(n,x')=a(n,x)$ for each $x'\in U$. Then for  $x'\in U$,
		$\rho_{a(nm,x)} (\varphi(x'))=\rho_{a(nm,x')} (\varphi(x'))=\varphi(\theta_{nm}(x'))=\varphi(\theta_{ m}(\theta_{ n}(x')))= \rho_{a(m,\theta_{ n}(x'))} (\varphi(\theta_{ n}(x')))=\rho_{a(m,\theta_{ n}(x')) }    \rho_{a(n ,x')} (\varphi(x')) = \rho_{a(n ,x)a(m,\theta_{ n}(x)) } (\varphi(x')) $.
		Essential freeness of $(Y,S,\rho)$ implies that $a(nm,x)=a(n,x)a(m,\theta_{n}(x))$.

Similarly, we can see the equation for the map $b$ holds.
\end{proof}
	
	\begin{lemma}
		In the situation of Definition 2.2, assume that $(X,P,\theta)$ and $(Y,S,\rho)$ are essentially free. Then
		$$b(a(m,x),\varphi(x))=m \mbox{ and }\,\, a(b(s,y),\varphi^{-1}(y))=s$$ for $ m\in P$, $x\in X$ and $s \in S$, $y\in Y$.
	\end{lemma}
	
	\begin{proof}
		We only show the first equation holds. From (2.1) and (2.2), one can see that $\theta_{m}(x)=\theta_{b(a(m,x),\varphi(x))}(x)$ for  $m\in P$ and $x\in X$. By the continuity of $a$ and $b$, this equation holds for    some open neighbourhood $ U $ of $x$.  Essential freeness of $(X,P,\theta)$  implies  that $b(a(m,x),\varphi(x))=m$.
\end{proof}

	\begin{corollary}
		In the situation of Definition 2.2, assume that $(X,P,\theta)$ and $(Y,S,\rho)$ are essentially free.
		For every  $x\in X$,  the map $a_{x}: m\in P\rightarrow a(m,x)\in S$ is a bijection with inverse $b_{\varphi(x)}:\, s\in S\rightarrow b(s,\varphi(x))\in P$, and  $ a_{x}(e)=e $.
		
	\end{corollary}
	\begin{proof}
		For $y\in Y$,	define $b_{y}: s\in S\rightarrow b(s,y)\in P$. Then it follows from Lemma 2.6 that $ a_{x}(b_{\varphi(x)}(s))  =s$ and $b_{\varphi(x)}(a_{x}(m)) =m$ for each $m\in P$ and $s\in S$. Thus $a_{x}$ and $b_{\varphi(x)}$ are inverse to each other.
		
  Remark that $\varphi(x)=\varphi(\theta_e(x))=\rho_{a(e,x)}(\varphi(x))$ for $x\in X$.  Choose an open neighbourhood $U$ of $x$ such that $a(e,x')=a(e,x)$ for each $x'\in U$. Then for  $x'\in U$,  $\rho_{a(e,x)}(\varphi(x'))=\rho_{a(e,x')}(\varphi(x'))=\varphi(x')$.  Essential freeness of $(Y,S,\rho)$ implies  $ a(e,x) =e$, thus $ a_{x}(e)=e $.
  \end{proof}

\begin{theorem}
		Two essentially free semigroup actions $(X,P,\theta)$ and $(Y,S,\rho)$ are continuously one-sided orbit equivalent if and only if two semi-groupoids $P\ltimes X$ and $S\ltimes Y$ are (topologically) isomorphic.
	\end{theorem}
	\begin{proof}
		Assume that $(X,P,\theta)$ and $(Y,S,\rho)$ are continuously  one-sided orbit equivalent and maps $\varphi$, $a$ and $b$ satisfy Definition 2.2.  Define $\Lambda: P\ltimes X\rightarrow S \ltimes Y$ and $\tilde{\Lambda}: S \ltimes Y\rightarrow P\ltimes X$ by $\Lambda(m,x)=(a(m,x),\varphi(x))$ and $\tilde{\Lambda}(s,y)=(b(s,y),\varphi^{-1}(y ))$. By Lemma 2.5 and Lemma 2.6, one can check that $\Lambda$ is an isomorphism as topological semi-groupoids with inverse isomorphism $\widetilde{\Lambda}$.
		
  Conversely, let  $\Lambda: P\ltimes X\rightarrow S \ltimes Y$ be an isomorphism  as topological semi-groupoids.  For $x\in X$, let $\Lambda(e,x)=(s,y)\in S\ltimes Y$. Since $(e,x)(e,x)=(e,x)$, it follows   that $\Lambda(e,x)\Lambda(e,x)=\Lambda(e,x)$. Consequently, $s=e$. Similarly, for each $y\in Y$,  one has that $\Lambda^{-1}(e,y)=(e,x)$ for some $x\in X$.  Hence, $\Lambda(\{e\}\times X)=\{e\}\times Y$. The spaces $X$ and $Y$ can be embedded into $P\ltimes X$ and $S\ltimes Y$, respectively, by identifying $(e,u)$ with $u$ for $u$ in $X$ or $Y$. Then the restriction $\varphi$ of $\Lambda$ to $X$ is a homeomorphism from $X$ onto $Y$.
		
		Define the map
		$a: (m,x)\in P\ltimes   X\rightarrow c_{\rho}\Lambda(m,x)\in S$, where $c_{\rho}(s,y)=s$ for $(s,y)\in S\ltimes Y$. Then $a$ is continuous. For $(m,x)\in P\ltimes X$, let $\Lambda(m,x)=(a(m,x),y)$ for $y\in Y$. Since $(m,x)(e,x)=(m,x)$, we have $\Lambda(m,x)$ and $\Lambda(e,x)$ are composable, which implies that $y=\varphi(x)$. Thus  $\Lambda(m,x)=(a(m,x),\varphi(x))$. Also since $(e, \theta_m(x))(m,x)=(m,x)$, we have $\Lambda(e,\theta_m(x))$ and $\Lambda(m,x)$ are composable. Thus $$\varphi(\theta_m(x))=\rho_{a(m,x)}(\varphi(x))$$ for $(m,x)\in P\ltimes X$.

Similarly, one can see that the map, $b: (s,y)\in S\ltimes Y\rightarrow c_{\theta}\Lambda^{-1}(s,y)\in P$, is continuous and satisfies that $\varphi^{-1}(\rho_{s}(y))=\theta_{b(s,y)}(\varphi^{-1}(y))$ for $(s,y)\in S\ltimes Y$,
		where $c_{\theta}(m,x)=m$ for $(m,x)\in P\ltimes X$. Hence  the maps $\varphi$, $a$ and $b$ give rise to the continuous  one-sided orbit equivalence of $(X,P,\theta)$ and $(Y,S,\rho)$.
	\end{proof}
	
	Let us compare conjugacy with continuous one-sided orbit equivalence.
	
	\begin{proposition} If two semigroup actions $(X,P,\theta)$ and $(Y,S,\rho)$ are conjugate, then they are continuously one-sided orbit equivalent. Moreover, if $X$ and $Y$ are connected and both of actions are essentially free, then the converse holds.
		
	\end{proposition}
	
	\begin{proof} Assume that $(X,P,\theta)$ and $(Y,S,\rho)$ are conjugate and maps $\varphi $ and $\alpha$ satisfy Definition 2.1 (i). Define $a(m,x)=\alpha(m)$ for $m\in P $, $x\in X$ and $b(s,y)=\alpha^{-1}(s)$ for $s\in S $,  $y\in Y$. Then $a$ and $b$ are continuous on their respective domains.  We can see that $\varphi$, $a$ and $b$ satisfy Definition 2.2, thus $(X,P,\theta)\sim_{csoe}(Y,S,\rho)$.

		Conversely, assume that $X$ and $Y$ are connected, and $(X,P,\theta)$ and $(Y,S,\rho)$ are essentially free and continuously one-sided orbit equivalent. Let $\varphi$, $a$ and $b$ be as in  Definition 2.2.  Then for every $m\in P$, $a|_{\{m\}\times X}$ is constant, thus we can define $\alpha(m)=a(m,x)$ for $m\in P$. It follows from Lemma 2.5, Lemma 2.6 and Corollary 2.7 that $\alpha:\; P\rightarrow S$ is a semigroup isomorphism satisfying that $\varphi(\theta_m(x))=\rho_{\alpha(m)}(\varphi(x))$ for each $m\in P$ and $x\in X$. Thus $(X,P,\theta)$ and $(Y,S,\rho)$ are conjugate.
	\end{proof}

	\section{Continuous Orbit Equivalence }
	
	Let $(X,P,\theta)$ be a semigroup action as in Section 2. Set $$X_{(m,n)}\, :=\{(x,y)\in X\times X \; | \; \theta_{m}(x)=\theta_{n}(y) \}\,\, \mbox{ for $(m,n)\in P\times P$},$$
$$X_{P,\theta}\, := \{(m,n,x,y)\in P\times P\times X\times X: (m,n)\in P\times P, (x,y)\in X_{(m,n)}\}.$$
	Then each $X_{(m,n)}$ is a nonempty compact subset in $X\times X$ and the latter is a topological subspace of the product topology space $ P\times P\times X \times X$.
Recall that two semigroup actions $(X,P,\theta)$ and $(Y,S,\rho)$ are orbit equivalent if there exists a homeomorphism $\varphi$ preserving each full orbit from $X$ onto $Y$. In this case, for  $(m,n)\in P\times P$ and  $(x,y)\in X_{(m,n)}$, there exist $s,t$ (depending on $m,n,x,y$) in $S$ such that $\rho_{s}(\varphi(x))=\rho_{t}(\varphi(y))$. Symmetrically, for  $(s,t)\in S\times S$ and $(u,v)\in Y_{(s,t)}$, there exist $m,n$ (depending on $s,t,u,v$) in $P$ such that $\theta_{m}(\varphi^{-1}(u))=\theta_{n}(\varphi^{-1}(v))$.
The following notion is a continuous version of orbit equivalence.
	
	\begin{definition}
		Two semigroup actions $(X,P,\theta)$ and $(Y,S,\rho)$ are \emph{continuously   orbit equivalent} (we write $(X,P,\theta)\sim_{ coe}(Y,S,\rho)$) if  there exist a homeomorphism $\varphi:\, X\rightarrow Y$,  continuous mappings $a_{1} ,b_{1} :\;  X_{P,\theta}\rightarrow S$    and  $a_{2} ,b_{2} : Y_{S,\rho} \rightarrow P $ such that  $$ \rho_{a_{1}(m,n,x,y)}(\varphi(x))=\rho_{b_{1} (m,n,x,y)}(\varphi(y)) \;\; \text{for} \;   (x,y)\in X_{(m,n)},   \eqno{(3.1)}$$
		$$\theta_{a_{2} (s,t,u,v)}(\varphi^{-1}(u))= \theta_{b_{2}(s,t,u,v)}(\varphi^{-1}( v))  \;\; \text{for} \;   (u,v)\in Y_{(s,t)}. \eqno{(3.2)}$$
	\end{definition}

	\begin{proposition} If $(X,P,\theta)$ and $(Y,S,\rho)$ are continuously one-sided orbit equivalent, then they are continuously orbit equivalent.
	\end{proposition}
	
	\begin{proof}
		Let $\varphi$, $a$ and $b$ be three  maps satisfying Definition 2.2.  For $m,n\in P$ and $(x,y)\in X_{(m,n)}$, define $a_{1}(m,n,x,y) =a(m,x)$ and $ b_{1}(m,n,x,y)=a(n,y)$. Then
		$a_{1},b_{1}: X_{P,\theta} \rightarrow S$    are continuous. Since $\theta_m(x)=\theta_n(y)$ for $(x,y)\in X_{(m,n)}$, it follows from (2.1) that $ \rho_{a_{1}(m,n,x,y) }(\varphi (x))= \rho_{b_{1}(m,n,x,y) }(\varphi (y))$.
		
	Similarly, we can  construct   continuous maps    $a_{2} , b_{2} : Y_{S,\rho} \rightarrow P $  satisfying (3.2). Thus $(X,P,\theta)$ and $(Y,S,\rho)$ are continuously orbit equivalent.
	\end{proof}

	Let $\mathcal{G}(X,P,\theta)$ be the transformation groupoid associated with $(X,P,\theta)$.  Clearly, each basic open subset of the form $\Sigma(U,m,n,V)$, denoted by $A$, of $\mathcal{G}(X,P,\theta)$  induces a homeomorphism $\alpha_A:\, x\in V\rightarrow (\theta_m|_{U})^{-1}(\theta_n(x))\in U$, where  $m,n\in P$ and  $U,V\subset X$ are open  such that $\theta_{m}|_{U} ,\theta_{n}|_{ V} $ are homeomorphisms and $\theta_{m} (U) =\theta_{n} ( V)$. Thus
$A=\{(\alpha_A(x), mn^{-1}, x):\,\, x\in V\}$.

	In the rest of this section, we characterize continuous orbit equivalence of semigroup actions in terms of the transformation  groupoids. Given two semigroup actions $(X,P,\theta)$ and $(Y,S,\rho)$, we let
 $G$ and $H$ be two related countable groups satisfying that  $P\subseteq G$, $S\subseteq H$ and the assumption in Section 2.

	\begin{lemma}
		For an essentially free semigroup action $(X,P,\theta)$, let $\alpha : \; U\rightarrow W$ be a homeomorphism between nonempty open subsets of $X$. Assume that there are continuous maps  $k,l :\; U\rightarrow P$ such that $\theta_{k(z)}(\alpha(z))=\theta_{l(z)}(z)$ for each $z\in U$. Then, for each $x\in U$, there is a unique $g\in G$  with the property that  there exist $k_{0}, l_{0} \in P$ and an open subset $V$ such that $g=k_{0}l_{0}^{-1}$, $x\in V \subseteq U$ and $\theta_{k_{0}}(\alpha(z))=\theta_{l_{0}}(z)$ for every $z\in V$.
		
		Moreover, if $k_1,l_1:\, U\rightarrow P$ are  another continuous maps such that $\theta_{k_1(z)}(\alpha(z))=\theta_{l_1(z)}(z)$ for all $z\in U$, then $k_1(x)l_1(x)^{-1}=k(x)l(x)^{-1}$ for each $x\in U$.
	\end{lemma}
	
	\begin{proof}
		For $x\in X$, let $k_{0}=k(x)$, $l_{0}=l(x)$ and $g=k_{0}l_{0}^{-1}$. Since $k,l :\; U\rightarrow P$ are continuous at $x$, there exists an open subset $V$ such that $x\in V\subseteq U$ and $k(z)=k(x)$,  $l(z)=l(x)$ for every $z\in V$. Thus $\theta_{k_{0}}(\alpha(z))=\theta_{l_{0}}(z)$ for each $z\in V$.
		
		For the uniqueness of $g$, assume that $g'\in G$, $k_{0}'$, $l_{0}'\in P$ and $V'$ is an open subset such that $g'=k_{0}'l_{0}'^{-1}$, $x\in V'\subseteq U$ and $\theta_{k_{0}'}(\alpha(z))=\theta_{l_{0}'}(z)$ for all $z\in V'$. Put $U'=V\cap V'$ and choose $p,q\in P$ such that $k_{0}^{-1}k_{0}'=pq^{-1}$. Thus $k_{0}p=k_{0}'q$, $\theta_{k_{0}p}(\alpha(z))=\theta_{l_{0}p}(z)$ and $\theta_{k_{0}'q}(\alpha(z))=\theta_{l_{0}'q}(z)$, which implies that
$\theta_{l_{0}p}(z)=\theta_{l_{0}'q}(z)  $ for each $z\in U'$. Essential  freeness implies that $l_{0}p=l_{0}'q$, thus $l_{0}^{-1}l_{0}'=pq^{-1}=k_{0}^{-1}k_{0}'$. Hence $ g=k_{0}l_{0}^{-1}=k_{0}'l_{0}'^{-1}=g' $.
		
		If $k_1,l_1:\, U\rightarrow P$ are  another continuous maps such that $\theta_{k_1(z)}(\alpha(z))=\theta_{l_1(z)}(z)$ for all $z\in U$. For $x\in X$, from the above proof, if let $k_0'=k_1(x)$, $l_0'=l_1(x)$ and $g'=k_{0}'l_{0}'^{-1}$, then there is an open subset $V'$ such that $x\in V' \subseteq U$ and $\theta_{k_{0}'}(\alpha(z))=\theta_{l_{0}'}(z)$ for all $z\in V'$. By  the above uniqueness, we have
		$g'=g$, i.e., $k_1(x)l_1(x)^{-1}=k(x)l(x)^{-1}$.
	\end{proof}

	\begin{lemma} Let $(X,P,\theta)$ and  $(Y,S,\rho)$ be continuously orbit equivalent and essentially free, and let $\varphi, a_{1}, b_{1},a_{2},b_{2}$ be as in Definition 3.1.	
		If  $ m_{1}n_{1}^{-1}=m_{2}n_{2}^{-1}$ and $ s_{1}t_{1}^{-1}=s_{2}t_{2}^{-1} $ for $m_{i},n_{i} \in P$, $s_{i},t_{i}\in S$ and $i=1,2$, then
		$$ a_{1}(m_{1},n_{1},x,y)b_{1}(m_{1},n_{1},x,y)^{-1}=a_{1}(m_{2},n_{2},x,y)b_{1}(m_{2},n_{2},x,y)^{-1}, $$
		$$ a_{2}(s_{1},t_{1},u,v)b_{2}(s_{1},t_{1},u,v)^{-1}=a_{2}(s_{2},t_{2},u,v)b_{2}(s_{2},t_{2},u,v)^{-1} $$
		for $(x,y) \in X_{(m_{1},n_{1})}\bigcap X_{(m_{2},n_{2})}$  and $(u,v)\in Y_{(s_{1},t_{1})}\bigcap Y_{(s_{2},t_{2})}   $.
		
	\end{lemma}
	
	\begin{proof}
		For $(x,y) \in X_{(m_{1},n_{1})}$, choose an open bisection $A=\Sigma(U_{1},m_{1},n_{1},V_{1})$ such that $(x,m_{1}n_{1}^{-1},y) \in A$. Let $\alpha_{A}$ be the homeomorphism from $V_1$ onto $U_1$ given by $A$, i.e., $\alpha_{A}(z)=(\theta_{m_{1}}|_{U_{1}})^{-1}(\theta_{n_{1}}(z))$ for $z\in V_{1}$. Then $\alpha_A(y)=x$, $A=\{ (\alpha_{A}(z),m_{1}n_{1}^{-1},z)| \; z\in V_{1}\}$ and $(\alpha_A(z),z)\in X_{(m_1,n_1)}$ for each $z\in V_1$.
		Define $k_{1}(z)=a_{1}(m_{1},n_{1},\alpha_{A}(z),z)$ and $l_{1}(z)=b_{1} (m_{1},n_{1},\alpha_{A}(z),z)$ for $z\in V_1$. Then both of them are continuous maps from $V_1$ into $S$. Thus we can choose an open neighbourhood $\widetilde{V_{1}}$ of $y$ such that  $\widetilde{V_{1}}\subseteq V_{1}$, $k_{1}(z)=k_{1}(y)$ and $l_{1}(z)=l_{1}(y)$ for each $z\in \widetilde{V_{1}}$. It follows from (3.1) that  $$\rho_{k_{1}(y)}(\varphi(\alpha_{A}(z)))=\rho_{l_{1}(y)}(\varphi(z))\mbox{ for $z\in \widetilde{V_{1}}$}.$$   Let $\widetilde{A}=\{(\alpha_{A}(z),m_{1}n_{1}^{-1},z)| \; z\in \widetilde{V_{1}} \}\subseteq A$.
		
		For $(x,y) \in   X_{(m_{2},n_{2})}$,  by a similar argument, there exists an open neighbourhood $\widetilde{V_{2}}$ of $y$ such that
		$$\rho_{k_2(y)}(\varphi(\alpha_{B}(z)))=\rho_{l_2(y)}(\varphi(z)) \mbox{ for $z\in \widetilde{V_{2}}$}$$
		where $k_2(y)=a_{1}(m_{2},n_{2}, x,y)$, $l_2(y)=b_{1}(m_{2},n_{2}, x,y)$, and $\alpha_B$ is the homeomorphism given by an open bisection $B= \Sigma(U_{2},m_{2},n_{2},V_{2})$ with $\widetilde{V_2}\subseteq V_2$ and
		$(x,m_{2}n_{2}^{-1},y) \in B$.
		Let $\widetilde{B}=\{(\alpha_{B}(z),m_{2}n_{2}^{-1},z)| \; z\in \widetilde{V_{2}}   \} \subseteq B$.
		
		Note that $\widetilde{A}\cap \widetilde{B}$ is a bisection containing $(x,m_1n_1^{-1},y)$ ($=(x,m_2n_2^{-1},y)$). Then there exists an open subset $V \subseteq\widetilde{V_{1}}\cap \widetilde{V_{2}}$ such that $y\in V $ and $\alpha_{A}(z)=\alpha_{B}(z)$  for each $z\in V$.
		Choose $p,q\in S$ such that $k_{1}(y)^{-1}k_{2}(y)=pq^{-1}$, thus $k_{1}(y)p=k_{2}(y)q$.  Hence it follows from the above equations that $$\rho_{l_{1}(y)p}( \varphi(z))=\rho_{k_{1}(y)p}(\varphi(\alpha_{A}(z)))=\rho_{k_{2}(y)q}(\varphi(\alpha_{B}(z)))=
		\rho_{l_{2}(y) q}(\varphi(z))$$
		for $z\in  V  $. Essential freeness of $(Y,S,\rho)$ implies that $l_{1}(y)p=l_{2}(y)q$, and thus $l_{1}(y)^{-1}l_{2}(y)=k_{1}(y)^{-1}k_{2}(y)$. Hence $  k_{1}(y)l_{1}(y)^{-1}=k_{2}(y)l_{2}(y)^{-1} $, i.e.,
		$ a_{1}(m_{1},n_{1},x,y)b_{1}(m_{1},n_{1},x,y)^{-1}=a_{1}(m_{2},n_{2},x,y)b_{1}(m_{2},n_{2},x,y)^{-1} $.
		
		Similarly, we can see that the equation for $a_2, b_2$ in the lemma holds.
	\end{proof}

\begin{remark}	From Lemma 3.4, both of the maps $a: \mathcal{G}(X,P,\theta)\rightarrow H$ and $b: \mathcal{G} (Y,S,\rho)\rightarrow G$, defined by $$a(x,mn^{-1},y)=a_{1}(m,n,x,y)b_{1}(m,n,x,y)^{-1}$$ and
 $$b(u,st^{-1},v)=a_{2}(s,t,u,v)b_{2}(s,t,u,v)^{-1}$$ are well-defined.  From the first paragraph of the proof for Lemma 3.4,  for $\gamma=(x,mn^{-1},y) \in \mathcal{G}(X,P,\theta) $, there exists an open neighbourhood of the form $A=\Sigma(U ,m ,n ,V )$ of $\gamma$ such that $\rho_{a_{1}(m,n,x,y)}(\varphi(u))=\rho_{b_{1}(m,n,x,y)}(\varphi(v))$ for all $(u,mn^{-1},v)\in A$.
		By (3.1), $\rho_{a_{1}(m,n,u,v)}(\varphi(u))=\rho_{b_{1} (m,n,u,v)}(\varphi(v))$ for all $(u,mn^{-1},v)\in A$.  Let $\alpha_A:\, v\in V\rightarrow (\theta_m)^{-1}(\theta_n(v))\in U$ be the canonical homeomorphism determined by $A$ and define $\alpha: \varphi(v)\in \varphi(V)\rightarrow \varphi(\alpha_A(v))\in\varphi(U)$. It follows from the continuity of $a_1$ and $b_1$ and  Lemma 3.3 that $a_1(m,n,x,y)b_1(m,n,x,y)^{-1}=a_1(m,n,u,v)b_1(m,n,u,v)^{-1}$ for each $(u,mn^{-1},v)\in A$. Hence $a(x,mn^{-1},y)=a(u,mn^{-1}, v)$ for $(u,mn^{-1},v)\in A$.  Consequently, $a$ is continuous. A similar argument shows that $b$ is also continuous. Thus we have the following continuous maps:
		$$\Psi:\,\,(x,mn^{-1},y)\in \mathcal{G}(X,P,\theta)\rightarrow (\varphi(x), a(x,mn^{-1},y), \varphi(y))\in   \mathcal{G}(Y,S,\rho)  $$
		and
		$$\widetilde{\Psi}:\, (u,st^{-1},v)\in  \mathcal{G}(Y,S,\rho)\rightarrow (\varphi^{-1}(u), b(u,st^{-1},v),\varphi^{-1}(v))\in \mathcal{G}(X,P,\theta).$$
\end{remark}

\begin{lemma}
		Let $(X,P,\theta)$ be essentially free and let $m_{i},n_{i} \in P$ for $i=1,2$. If there exists an nonempty open set $U\subset X$ such that for each $x\in U$, there exists $y\in X$ satisfying
		$(x,y) \in X_{(m_{1},n_{1})}\bigcap X_{(m_{2},n_{2})}$, then
		$m_{1}n_{1}^{-1}=m_{2}n_{2}^{-1}$.
		
	\end{lemma}
	
	\begin{proof}
		Let  $n_{1}^{-1}n_{2}=pq^{-1}$ for $p, q\in P$. Then $n_{1}p=n_{2}q$. For each $x\in U$, by assumption, there is $y\in X$ such that $\theta_{m_i}(x)=\theta_{n_i}(y)$ for $i=1,2$, thus  $\theta_{m_{1}p}(x)= \theta_{n_{1}p}(y) = \theta_{n_{2}q}(y)= \theta_{m_{2}q}(x)$ for all $x\in U$.
		Essential freeness implies that $m_{1}p=m_{2}q$, then $m_{1}n_{1}^{-1}=m_{2}n_{2}^{-1}$.	
	\end{proof}

	\begin{lemma}
		The   mappings  $a$ and $b$ defined in Remark 3.5 are    cocycles on  $\mathcal{G}(X,P,\theta)$ and $\mathcal{G} (Y,S,\rho)$, respectively.
		
	\end{lemma}
	\begin{proof}
		Since the argument to deal with $a$ and $b$ is similar, we only consider the map $a$. Let $\gamma_{1}=(x_{0},m_{1}n_{1}^{-1},y_{0}), \gamma_{2}=(y_{0},m_{2}n_{2}^{-1},z_{0})\in \mathcal{G}(X,P,\theta)$ and write $\eta=\gamma_{1}\gamma_{2}=(x_{0},m_{1}n_{1}^{-1}m_{2}n_{2}^{-1},z_{0})$. Choose $p,q\in P$ satisfying that $n_{1}^{-1}m_{2}=pq^{-1}  $. Then $n_{1}p=m_{2}q$ and $\eta=(x_{0},m_{1}p(n_{2}q)^{-1},z_{0})$. By Remark 3.5, there exist open bisections $A=\Sigma(U_{1},m_{1},n_{1},V_{1})$, $B=\Sigma(U_{2},m_{2},n_{2},V_{2 })$ and $C=\Sigma(W_{1},m_{1}p,n_{2}q,W_{2})$ such that  $  \gamma_{1}\in A , \gamma_{2}\in B, \eta\in C$ and the following statements hold:
		\begin{enumerate}
			\item[(i)]  $\rho_{a_{1}(m_{1},n_{1},x_{0},y_{0})}(\varphi(x))=\rho_{b_{1}(m_{1},n_{1},x_{0},y_{0})}(\varphi(y))$;
			\item[(ii)] $\rho_{a_{1}(m_{2},n_{2},y_{0},z_{0})}(\varphi(u))=\rho_{b_{1}(m_{2},n_{2},y_{0},z_{0})}(\varphi(v))$;
			\item[(iii)] $\rho_{a_{1}(m_{1}p,n_{2}q,x_{0},z_{0})}(\varphi(z))=\rho_{b_{1}(m_{1}p,n_{2}q,x_{0},z_{0})}(\varphi(w))$,
		\end{enumerate}
		for $(x,m_{1}n_{1}^{-1},y)\in A$, $(u,m_{2}n_{2}^{-1},v)\in B$ and $(z,m_{1}p(n_{2}q)^{-1},w)\in C$. By the continuity of multiplication on $\mathcal{G}(X,P,\theta)^{(2)}$, we can assume that $V_{1}=U_{2}$ and  $AB\subset C$. 
		
		For each $z\in \varphi(U_{1})$, choose $\alpha=(x,m_{1}n_{1}^{-1},y)\in A$ and $ \beta=(y,m_{2}n_{2}^{-1},v)\in B $ such that $z=\varphi(x)$ and $\alpha\beta= (x,m_{1}p(n_{2}q)^{-1},v)\in C$. It follows from (i) and (ii) that $\rho_{a_{1}(m_{1},n_{1},x_{0},y_{0})}(z)=\rho_{b_{1}(m_{1},n_{1},x_{0},y_{0})}(\varphi(y))$  and $\rho_{a_{1}(m_{2},n_{2},y_{0},z_{0})}(\varphi(y))=\rho_{b_{1}(m_{2},n_{2},y_{0},z_{0})}(\varphi(v))$.
		Let $s,t\in S$ with $ b_{1}(m_{1},n_{1},x_{0},y_{0})s= a_{1}(m_{2},n_{2},y_{0},z_{0})t$. Thus $\rho_{a_{1}(m_{1},n_{1},x_{0},y_{0})s}(z)=$   $\rho_{b_{1}(m_{2},n_{2},y_{0},z_{0})t}(\varphi(v))  $.
		From (iii),  $\rho_{a_{1}(m_{1}p,n_{2}q,x_{0},z_{0})}(z)=\rho_{b_{1}(m_{1}p,n_{2}q,x_{0},z_{0})}(\varphi(v))  .$
		By Lemma 3.6,  $a_{1}(m_{1},n_{1},x_{0},y_{0})s (b_{1}(m_{2},n_{2},y_{0},z_{0})t)^{-1}$ \newline $=a_{1}(m_{1}p,n_{2}q,x_{0},z_{0})b_{1}(m_{1}p,n_{2}q,x_{0},z_{0})^{-1}$. Thus
		$$a(x_{0},m_{1}n_{1}^{-1},y_{0})a(y_{0},m_{2}n_{2}^{-1},z_{0})=a(x_{0},m_{1}n_{1}^{-1}m_{2}n_{2}^{-1},z_{0}),$$ which implies that $a$ is a cocycle.
	\end{proof}
	
	\begin{lemma}   The   mappings  $a$ and $b$ defined in Remark 3.5 satisfy that
		$$b(\varphi(x),a(x,mn^{-1},y),\varphi(y))=mn^{-1}, $$
		$$a(\varphi^{-1}(u),b(u,st^{-1},v),\varphi^{-1}(v))=st^{-1}, $$
		for every $(x,mn^{-1},y)\in \mathcal{G}(X,P,\theta) $ and $(u,st^{-1},v)\in \mathcal{G}(Y,S,\rho)$.
	\end{lemma}

	\begin{proof}
		For an element $(x_{0},mn^{-1},y_{0})\in \mathcal{G}(X,P,\theta)$, let $a_{1}(m,n,x_{0},y_{0})=s$ and  $b_{1}(m,n,x_{0},y_{0})=t$. Then $a(x_{0},mn^{-1},y_{0})=st^{-1}$ and $\Psi(x_{0},mn^{-1},y_{0})=(\varphi(x_{0}),st^{-1},\varphi(y_{0}))\in  \mathcal{G}(Y,S,\rho)$.
		From Remark 3.5, there exist    open bisections $ A=\Sigma(U_{1} ,m ,n ,V_{1} ) $, $B=\Sigma(U_{2} ,s ,t ,V_{2} )$ satisfying   $(x_{0},mn^{-1},y_{0})\in A$, $(\varphi(x_{0}),st^{-1},\varphi(y_{0}))\in B $,
		$\rho_{a_{1}(m,n,x_{0},y_{0})}(\varphi(x))=\rho_{b_{1}(m,n,x_{0},y_{0})}(\varphi(y)) $ for each $(x,mn^{-1},y)\in A$ and $ \theta_{a_{2}(s,t,\varphi(x_{0}),\varphi(y_{0}))}(\varphi^{-1}(u))=  \theta_{b_{2}(s,t,\varphi(x_{0}),\varphi(y_{0}))}(\varphi^{-1}(v))$  for each $(u,st^{-1},v)\in B$.
		By the continuity of $\varphi$ at $x_0$ and $y_0$ and that of $\Psi$ at $(x_{0},mn^{-1},y_{0})$ , we can assume that $\varphi(U_{1})\subseteq U_{2}$, $\varphi(V_{1})\subseteq V_{2}$ and $\Psi(A)\subseteq B$.

		For each $u\in U_{1} $,   there exists $v\in V_1$ such that $ \alpha=(u,mn^{-1},v)\in A $. Then by assumption, we have $ \rho_{a_{1}(m,n,x_{0},y_{0})}(\varphi(u))=\rho_{b_{1}(m,n,x_{0},y_{0})}(\varphi(v)) $ and
		$   (\varphi(u),st^{-1},\varphi(v))\in B$. Thus $\theta_{a_{2}(s,t,\varphi(x_{0}),\varphi(y_{0}))}(u)=  \theta_{b_{2}(s,t,\varphi(x_{0}),\varphi(y_{0}))}( v)$.
		Also since $ \theta_{m}(u)=\theta_{n}(v)$,  it follows from Lemma 3.6 that $$a_{2}(s,t,\varphi(x_{0}),\varphi(y_{0}))b_{2}(s,t,\varphi(x_{0}),\varphi(y_{0}))^{-1}=mn^{-1}.$$
		Thus $b(\varphi(x_{0}),a(x_{0},mn^{-1},y_{0}),\varphi(y_{0}))=mn^{-1} $.
		
		By a similar argument, one can see that the other equation holds.
\end{proof}

	\begin{theorem}
		Let $(X,P,\theta)$ and  $(Y,S,\rho)$ be two essentially free semigroup actions. If
		$(X,P,\theta)\sim_{coe}(Y,S,\rho)$, then
		$\mathcal{G}(X,P,\theta)$ and $\mathcal{G}(Y,S,\rho)$ are  isomorphic as \'{e}tale groupoids.
	\end{theorem}
	
	\begin{proof}
		Let $\varphi$, $a_{1}$, $b_{1}$, $a_{2}$ and $b_{2} $ be as in Definition 3.1. Let $\Psi$ and $\widetilde{\Psi}$ be the continuous maps defined in Remark 3.5. From Lemma 3.7, $\Psi$ and $\widetilde{\Psi}$ are continuous groupoid homomorphisms, and from Lemma 3.8, they are inverse to each other. Hence $\mathcal{G}(X,P,\theta)$ and $\mathcal{G}(Y,S,\rho)$ are  isomorphic as \'{e}tale groupoids.
	\end{proof}	
	
	\begin{proposition}
		Assume that  $X$ and $Y$ are totally disconnected spaces.  If $\mathcal{G}(X,P,\theta)$ and $\mathcal{G}(Y,S,\rho) $ are  isomorphic, then 	$(X,P,\theta)\sim_{coe}(Y,S,\rho)$.
	\end{proposition}	
	\begin{proof}
		Assume that $ \Lambda: \mathcal{G}(X,P,\theta)\rightarrow\mathcal{G}(Y,S,\rho) $ is an isomorphism.  Let $\varphi$ be the restriction of $\Lambda$ to $X$, and let $a(\gamma)=c_{\rho}\Lambda(\gamma)$ and $b(\eta)=c_{\theta}\Lambda^{-1}(\eta)$, where $c_{\theta}$ and $c_{\rho}$ are the canonical cocycles on $\mathcal{G}(X,P,\theta)$ and $\mathcal{G}(Y,S,\rho)$, respectively. Then $\varphi$ is a homeomorphism from $X$ onto $Y$, $a, b$ are continuous cocycles on their respective domains. Moreover, $\Lambda(x,g,y)=(\varphi(x),a(x,g,y),\varphi(y))$ and $\Lambda^{-1}(u,h,v)=(\varphi^{-1}(u),b(u,h,v),\varphi^{-1}(v))$.
		Let $(m,n)\in P\times P$ be arbitrary.
		
		For $(x,y)\in X_{(m,n)}$, let $\gamma=(x,mn^{-1},y)\in \mathcal{G}(X,P,\theta)$. Since $a$ is continuous, there is an open bisection $A=\Sigma(U,m,n,V)$ such that $\gamma\in A$ and $a(\gamma)=a(\gamma')$ for each $\gamma'\in A$. Let $\alpha_{A}$ be the canonical homeomorphism given  by $A$, i.e., $\alpha_{A}(z)=(\theta_{m}|_{U})^{-1}(\theta_{n}(z))$, $z\in V$. Then
		$A=\{ (\alpha_{A}(z),mn^{-1},z) \;|\,z\in V\}$, thus $\Lambda(A)=\{ (\varphi(\alpha_{A}(z)),a(\gamma),\varphi(z))\;|\;z\in V\}$ is an open bisection. Choose an bisection of the form $\widetilde{B}=\Sigma(W,s,t,T)$ such that $\Lambda(\gamma)\in \widetilde{B}$ and $\widetilde{B}\subseteq \Lambda(A)$, and let $B=\Lambda^{-1}(\widetilde{B})\subseteq A$. Then $a(\gamma)=st^{-1}$ for $s,t\in S$ and there exists an open subset $V'\subseteq V$ such that $y\in V'$ and $B=\{ (\alpha_{A}(z),mn^{-1},z) \;|\,z\in V'\}=\Sigma(U',m,n,V')$, where $U'=\alpha_A(V')$.
		Thus $\rho_{s}(\varphi(\alpha_{A}(z)))=\rho_{t}(\varphi(z))$ for all $z\in V'$, so  $\rho_{s}(\varphi(u))=\rho_{t}(\varphi(v))$ for $(u,v)\in (U'\times V')\cap X_{(m,n)}$.
		
		Above all, we conclude that, for each $(x,y)\in X_{(m,n)}$, there exist $s,t\in S$ and open neighbourhoods $U_{x}=U'$ of $x$ and  $V_{y}=V'$ of $y$ such that $a(x,mn^{-1},y)=st^{-1}$, $\theta_{m}(U_{x})=\theta_{n}(V_{y})$, $\theta_{m}|_{U_{x} },\theta_{n}|_{ V_{y}} $  are injective, and  $  \rho_{s}(\varphi(u))$ $=\rho_{t}(\varphi(v))$ for each $(u,v)\in  (U_{x}\times V_{y})\cap X_{(m,n)}$.
		
		For each $y\in X$, since $X$ is compact and $\theta_{m},\theta_{n}$ are surjective local homeomorphisms, there exist finite elements in $X$, denoted by $x_{1},x_{2},\cdots,x_{k}$, such that $(x_{i},y)\in X_{(m,n)}$ for $i=1,2,\cdots,k$. From the above all, for each $i$ with $1\leq i\leq k$, there exist two elements, denoted by $\widetilde{a}_1(m,n,x_{i},y)$ and $\widetilde{b}_1(m,n,x_{i},y)$, in $S$ and open subsets $V_{y}\ni y$ and $U_{x_{i}}\ni x_{i}$ such that
		$a(x_{i},mn^{-1},y)=\widetilde{a}_1(m,n,x_{i},y) \widetilde{b}_1(m,n,x_{i},y)^{-1}$, $\theta_{m}(U_{x_{i}})=\theta_{n}(V_{y})$, $\theta_{m}|_{U_{x_{i}} }$ and $\theta_{n}|_{ V_{y}} $  are injective, and
		$\rho_{\widetilde{a}_1(m,n,x_{i},y)}(\varphi(u)) =\rho_{\widetilde{b}_1(m,n,x_{i},y)}(\varphi(v))$ for $(u,v)\in (U_{x_{i}}\times V_{y})\cap X_{(m,n)}$.
		Due to $X$ is  totally disconnected, the $V_{y}$ above can be assumed to be a clopen subset of $X$. From the compactness of $X$, there exist $y_{i},\,x_{ij}\in X$ and clopen subsets $V_i$ and $U_{ij}$ of $X$ for $i=1,2,\cdots,l$, $j=1,2,\cdots, k_i$ satisfying the following conditions:
		\begin{enumerate}
			\item[(i)] $X= \mathop{\bigcup}\limits_{i=1}^{l}V_{i}$ is the disjoint union of $ V_{i}'s$, and $y_i\in V_i, x_{ij}\in U_{ij}$ for $i=1,2,\cdots,l$, $j=1,2,\cdots, k_i$;
			
			\item[(ii)]   $\theta_{m}(x_{ij})=\theta_{n}(y_{i})$,  $\theta_{m}(U_{ij})=\theta_{n}(V_i)$, $\theta_{m}|_{U_{ij}},\theta_{n}|_{ V_{i}} $  are injective for $i=1,2,\cdots,l$, $j=1,2,\cdots,k_{i}$;
			
			\item[(iii)] $  \mathop{\bigcup}\limits_{i=1}^{l}\mathop{\bigcup}\limits_{j=1}^{k_{i}}((U_{ij}\times V_{i})\cap X_{(m,n)})=X_{(m,n)} $;
			\item[(iv)] there exist $  \widetilde{a}_1(m,n,x_{ij},y_{i}), \widetilde{b}_1(m,n,x_{ij},y_{i})\in S$  with $a(x_{ij},mn^{-1},y_i)= \widetilde{a}_1(m,n,x_{ij},y_{i})\widetilde{b}_1(m,n,x_{ij},y_{i})^{-1}$ and $$ \rho_{ \widetilde{a}_1(m,n,x_{ij},y_{i})}(\varphi(u)) =\rho_{ \widetilde{b}_1(m,n,x_{ij},y_{i})}(\varphi(v))$$   for $(u,v)\in (U_{ij}\times V_{i})\cap X_{(m,n)}$, $i=1,2,\cdots,l$, $j=1,2,\cdots,k_{i}$.
		\end{enumerate}	
		
		We define two maps 	$a_{1}$ and $b_{1}$ from $X_{P,\theta}= \cup_{(m,n)\in P\times P} ( \{(m,n)\} \times X_{(m,n)} )$ into $S$   by $a_{1}(m,n,u,v)= \widetilde{a}_1(m,n,x_{ij},y_{i})$ and $b_{1}(m,n,u,v)= \widetilde{b}_1(m,n,x_{ij},y_{i})$ for $ (  u,v)\in    (U_{ij}\times V_{i})\cap X_{(m,n)}  $. Then $a_{1} ,b_{1}$ are continuous mappings satisfying the equation $(3.1)$.
		Similarly, we can construct continuous maps  $a_{2} ,b_{2}$ satisfying (3.2). Thus $(X,P,\theta)\sim_{coe}(Y,S,\rho)$.
\end{proof}
	
	Recall that an \'{e}tale groupoid $\mathcal{G}$ is \emph{topologically principal} if $\{u\in \mathcal{G}^{(0)}:\,\, \mathcal{G}^u_u=\{u\}\}$ is dense in $\mathcal{G}^{(0)}$, where $\mathcal{G}_{u}^{u}=\{\gamma\in \mathcal{G},r(\gamma)=d(\gamma)=u\} $.
	From \cite{BNRSW,Re2}, if $\mathcal{G}$ is topologically principal, then $C_0(\mathcal{G}^{(0)})$ is a Cartan subalgebra of $C_r^*(\mathcal{G})$. Furthermore,  two topologically principal \'{e}tale groupoids $\mathcal{G}$ and $\mathcal{H}$ are isomorphic if and only if there exists a $C^*$-isomorphism $\Phi$ from $C^*_r(\mathcal{G})$ onto $C^*_r(\mathcal{H})$ such that $\Phi(C_0(\mathcal{G}^{(0)}))=C_0(\mathcal{H}^{(0)})$.
	From \cite[Proposition 7.5]{BCFS},  a semigroup   action  $(X,P,\theta)$ is essentially free  if and only if $\mathcal{G}(X,P,\theta)$   is topologically principal.  By Theorem 3.9 and Proposition 3.10, we have the following   result.
	
	\begin{corollary}
		Assume that $X$ and $Y$ are totally disconnected and that $(X,P,\theta)$ and  $(Y,S,\rho)$ are essentially free. Then following statements are equivalent:
		\begin{enumerate}
			\item[(i )] $ (X,G, \theta )\sim_{coe} (Y,H, \rho )$;
			\item[(ii )]  $\mathcal{G}(X,P,\theta)$ and $\mathcal{G}(Y,S,\rho)$ are  isomorphic as \'{e}tale groupoids;
			\item[(iii)] there is a $C^*$-isomorphism $\Phi$ from $C_r^*(\mathcal{G}(X,P,\theta))$   onto $C_r^*(\mathcal{G}(Y,S,\rho))$ such that $\Phi(C(X))=C(Y)$.
		\end{enumerate}

	\end{corollary}

	\begin{example}
		Let $\mathbb{N}_{0}$ be the additive semigroup of all non-negative integers. For a finite set $A$, consider the set $A^{\mathbb{N}_0}$ consisting of all maps from $\mathbb{N}_0$ to $A$. Equipped each factor $A$ of $A^{\mathbb{N}_0}$ with the discrete topology and $A^{\mathbb{N}_0}$ with the associated   product topology, $A^{N_0}$ is compact and totally disconnected space. Let $\sigma:\, A^{\mathbb{N}_{0}}\rightarrow A^{\mathbb{N}_{0}}$ be the shift transformation defined by $$\sigma(x)(i)=x(i+1) \mbox{ for $x\in A^{\mathbb{N}_0}$ and $i\in \mathbb{N}_{0}$}.$$
		
		A one-sided shift space is a closed, and hence compact, subset $X$ of $A^{\mathbb{N}_{0}}$ such that $X$ is invariant by the shift transformation $\sigma$, i.e., $\sigma(X)=X$.  In this case, let $\sigma_{X}$ denote the restriction of $\sigma$ to $X$. The shift map $\sigma_X$ is a local homeomorphism if and only if $X$ is a shift of finite type, in which case $\sigma_X^n$ is a local homeomorphism for all $n\in \mathbb{N}_0$ (\cite[2.2]{CEOR}), thus we have a semigroup action $(X, \mathbb{N}_0, \sigma_X)$ in a natural way.
		
		Following \cite{Ma1}, the authors in \cite{CEOR} introduced the notion of continuous orbit equivalence for one-sided shift spaces, in which they call two one-sided shift spaces $X$ and $Y$  are continuously orbit equivalent if there exist  a homeomorphism $\varphi:X\rightarrow Y$ and continuous maps $k,l: X\rightarrow \mathbb{N}_{0}$, $k',l': Y\rightarrow \mathbb{N}_{0}$ such that $\sigma_{Y}^{k(x)}(\varphi(\sigma_{X}(x)))=\sigma_{Y}^{l(x)}(\varphi( x ))$ and $\sigma_{X}^{k'(y)}(\varphi^{-1}(\sigma_{Y}(y)))=\sigma_{X}^{l'(y)}(\varphi^{-1}( y ))$ for $x\in X$ and $y\in Y$. Moreover, they also proved that two one-sided shift spaces of finite type, $X$ and $Y$, are continuously orbit equivalent if and only if their associated groupoids $\mathcal{G}_X$ and $\mathcal{G}_Y$ are isomorphic. The following proposition shows that the notion of continuous orbit equivalence in \cite{CEOR} and that of semigroup actions for one-sided shift spaces of finite type are consistent.
\end{example}

	\begin{proposition}
		Two one-sided shift spaces of finite type $X$ and $Y$  are continuously orbit equivalent	if and only if semigroup actions $(X,\mathbb{N}_{0},\sigma_{X})$ and  $(Y,\mathbb{N}_{0},\sigma_{Y})$ are continuously orbit equivalent.
	\end{proposition}
	\begin{proof}
		Assume that $X$ and $Y$  are continuously orbit equivalent via a homeomorphism $\varphi$ and continuous maps $k,l,k',l'$ as above. For $x\in X$,
		define  $k^{n}(x)= \sum\limits_{i=0}^{n-1}k(\sigma_{X}^{i}(x))$ and $l^{n}(x)=\sum\limits_{i=0}^{n-1}l(\sigma_{X}^{i}(x))$ for $n\geq 1$, and $k^0(x)=l^0(x)=0$.
		Then $$\sigma_{Y}^{k^{n}(x)}(\varphi(\sigma^{n}_{X}(x)))=\sigma_{Y}^{l^{n}(x)}(\varphi( x ))$$ for $x\in X$ and $n\geq 0$ (\cite[Lemma 5.1]{Ma1}). Thus, for $(m,n)\in \mathbb{N}_0\times \mathbb{N}_0$ and $\sigma_X^m(x)=\sigma_X^n(y)$, one can see that
		$\sigma_{Y}^{k^{m}(x)+l^{n}(y)}(\varphi(y))=\sigma_{Y}^{k^{n}(y)+l^{m}(x)}(\varphi( x ))$.

		Define $a_{1}(m,n,x,y)=l^{m}(x)+k^{n}(y)$, $b_{1}(m,n,x,y)=k^{m}(x)+l^{n}(y)$ for $(m,n,x,y)\in \cup_ {(m,n)\in \mathbb{N}_0\times \mathbb{N}_0}( \{(m,n)\}  \times X_{(m,n)} )$. Then $a_1$ and $b_1$ are continuous and satisfy  that $\sigma_{Y}^{a_{1}(m,n,x,y) }(\varphi(x))=\sigma_{Y}^{b_{1}(m,n,x,y)}(\varphi( y ))$.
		By a similar argument, we can construct  continuous maps  $a_{2}$ and $b_{2}$ on $  \cup_ {(s,t)\in \mathbb{N}_0\times \mathbb{N}_0}( \{(s,t)\}  \times Y_{(s,t)} ) $  satisfying that $\sigma_{X}^{a_{2}(s,t,x,y) }(\varphi^{-1}(x))=\sigma_{X}^{b_{2}(s,t,x,y)}(\varphi^{-1}( y ))$. Hence $(X,\mathbb{N}_{0},\sigma_{X})$ and  $(Y,\mathbb{N}_{0},\sigma_{Y})$ are continuously orbit equivalent.

		Assume that $(X,\mathbb{N}_{0},\sigma_{X})$ and  $(Y,\mathbb{N}_{0},\sigma_{Y})$ are continuously orbit equivalent and $\varphi ,a_{1} ,b_{1} ,a_{2} ,b_{2} $ satisfy Definition 3.1.
		Let $k(x)=b_{1}(1,0,x,\sigma_{X}(x) )$ and $l(x)=a_{1}(1,0,x,\sigma_{X}(x) )$ for $x\in X$,   $k'(y)=b_{2}(1,0,y,\sigma_{Y}(y) )$ and $l'(y)=a_{2}(1,0,y,\sigma_{Y}(y) )$ for $y\in Y$.
		Then $k,l: X\rightarrow \mathbb{N}_{0}$ and $k',l': Y\rightarrow \mathbb{N}_{0}$ are continuous maps such that $\sigma_{Y}^{k(x)}(\varphi(\sigma_{X}(x)))=\sigma_{Y}^{l(x)}(\varphi( x ))$ and $\sigma_{X}^{k'(y)}(\varphi^{-1}(\sigma_{Y}(y)))=\sigma_{X}^{l'(y)}(\varphi^{-1}( y ))$  for $x\in X$ and $y\in Y$.  Therefore  $X$ and $Y$  are continuously orbit equivalent.
	\end{proof}

	\section{  semigroup actions by homeomorphisms }

	Let  $(X,P,\theta)$ be a semigroup action and $G$ a countable group containing $P$ as in Section 2. In this section, we further assume that each map $\theta_m$  is a homeomorphism, in other words, $(X,P,\theta)$ is a semigroup action by homeomorphisms. Under this situation, we can construct a group action  $(X,G,\widetilde{\theta})$ and discuss the relationship between   continuous orbit  equivalence of semigroup actions and that of group actions.
	
	For  each $g\in G$, it follows from the assumption that there exist $m,n\in P$ such that $g=mn^{-1}$. Define $$\tilde{\theta}_{g}(x)= \theta_{n}^{-1}(\theta_{m}(x))\; \text{for}  \;x\in X.$$
	To see that $\widetilde{\theta}_g$ is well-defined, if  $g=m_{1}n_{1}^{-1}=m_{2}n_{2}^{-1}$ for $m_i,n_i \in P$ and  $i=1,2$, then we can choose $p,q \in P$ such that $m_{2}^{-1}m_{1}(=n_2^{-1}n_1)=pq^{-1}$. Thus $m_{2}p=m_{1}q$ and $n_{2}p=n_{1}q$. For $x\in X$, we have
	$$\theta_{n_{2}p}(\theta_{n_{1}}^{-1}(\theta_{m_{1}}(x)))=\theta_{n_{1}q}(\theta_{n_{1}}^{-1}(\theta_{m_{1}}(x)))=\theta_{m_{1}q}  (x)=\theta_{m_{2}p}(x)=\theta_{n_{2}p}(\theta_{n_{2}}^{-1}(\theta_{m_{2}}(x))).$$
	Since $\theta_{n_2p}$ is a homeomorphism, we have $\theta_{n_{1}}^{-1}(\theta_{m_{1}}(x))=\theta_{n_{2}}^{-1}(\theta_{m_{2}}(x))$. Hence $\theta_{n_{1}}^{-1}\theta_{m_{1}}  =\theta_{n_{2}}^{-1}\theta_{m_{2}}$.
	
	Clearly, $\tilde{\theta}_{m} =\theta_{m}$ and $ \tilde{\theta}_{m^{-1}} =\theta_{m}^{-1}  $ for each $ m\in P $. From the above, for $x,y \in X$ and $m_{i},n_{i} \in P$, $i=1,2$, if $m_{1}n_{1}^{-1}=m_{2}n_{2}^{-1}$, then   $\theta_{m_{1}}(x)=\theta_{n_{1}}(y)$ if and only if   $\theta_{m_{2}}(x)=\theta_{n_{2}}(y)$.

	\begin{lemma}
		The map $\tilde{\theta}:\, g\rightarrow \widetilde{\theta}_g$ is a (right) group action of $G $ on $X$.
	\end{lemma}
	\begin{proof}
		Given $g,h\in G$, we let $g=ab^{-1}$ and $h=cd^{-1}$ for $a,b,c,d\in P$. Choose $m,n\in P$ such that $b^{-1}c=mn^{-1}$. Thus $ \theta_{n}^{-1}\theta_{m}=\theta_{c}\theta_{b}^{-1}$. Consequently, for each $x\in X$, we have
		$$\begin{aligned}
			\tilde{\theta}_{gh}(x)&=\tilde{\theta}_{a(b^{-1}c)d^{-1}}(x) =\tilde{\theta}_{amn^{-1}d^{-1}}(x)\\& =\theta_{dn}^{-1}\theta_{am}(x)  =\theta_{d}^{-1} \theta_{n}^{-1}\theta_{ m}\theta_{a } (x)\\&=\theta_{d}^{-1}\theta_{c}\theta_{b}^{-1}\theta_{a}(x) =\tilde{\theta}_{h}\tilde{\theta}_{g}(x).
		\end{aligned}$$	
		Thus	$ \tilde{\theta}$ is a right action of   $ G $ on  $X$.
\end{proof}

	The transformation groupoid  $X\rtimes_{\widetilde{\theta}} G$ associated to the above group action $(X,G,\widetilde{\theta})$   is given by the set $X\times G$ with the product topology,  multiplication $(x,g)(y,h)=(x,gh)$ if $y=\widetilde{\theta}_g(x)$, and inverse $(x,g)^{-1}=(\widetilde{\theta}_g(x),g^{-1})$. This groupoid is \'{e}tale and its unit space equals $X$ by identifying  $(x,e)$    with $x$.
	It is well-known that the reduced groupoid $C^*$-algebra  $C_{r}^{*}(X\rtimes_{\widetilde{\theta}} G)$ is isomorphic to the reduced crossed product $C^{*}$-algebra $C(X)\rtimes_{\widetilde{\theta},r}G$.
	From \cite{CRST}, when $G=\mathbb{Z}$,   the associated groupoid $\mathcal{G}(X,\mathbb{Z},\theta)$ of   Deaconu-Renault system $(X, \mathbb{Z},\theta)$ is isomorphic to the transformation groupoid $X\rtimes_{\widetilde{\theta}}\mathbb{Z}$, which induces an isomorphism $\Phi:\,C^{*}_r(\mathcal{G}(X,\mathbb{Z},\theta))\rightarrow C(X)\rtimes_{\widetilde{\theta},r}\mathbb{Z}$.  Similarly, we have the following result.

	\begin{proposition} 	The map
		$\Lambda:\;(x,g,y) \in \mathcal{G}(X,P,\theta)\mapsto (x,g)\in X\rtimes_{\widetilde{\theta}}  G$  is an  \'{e}tale groupoid isomorphism. Moreover, it induces  a $C^*$-isomorphism $ \Phi$ from $C_{r}^{*}(\mathcal{G}(X,P,\theta) )$ onto $ C(X)\rtimes_{\tilde{\theta},r}G$ such that $\Phi(C(X))= C(X)$.
	\end{proposition}
	\begin{proof}
		We only prove that $\Lambda$ is an \'{e}tale groupoid isomorphism. One can see that $\Lambda$ is an algebraic  (groupoid) isomorphism from $\mathcal{G}(X,P,\theta)$ onto $ X\rtimes_{\widetilde{\theta}}  G$ with inverse $\Lambda^{-1}$, defined by $\Lambda^{-1}(x,g)=(x,g,\tilde{\theta}_{g}(x))$ for $(x,g)\in X\rtimes_{\widetilde{\theta}}  G$.

		Given $(x,g,y)\in \mathcal{G}(X,P,\theta)$, we assume that $g=ab^{-1}$ and $\theta_{a}(x)=\theta_{b}(y)$ for $a,b\in P$. For an arbitrary open subset $U\subseteq X$ with $x\in U$, the set  $\Sigma(U,a,b,\theta_{b}^{-1}(\theta_{a}(U)))$ is an open neighbourhood of  $(x,g,y)$ in $ \mathcal{G}(X,P,\theta) $ and $\Lambda(\Sigma(U,a,b,\theta_{b}^{-1}(\theta_{a}(U)))=U\times\{g\}$. Thus $\Lambda$ is continuous at $(x,g,y)$. By a  similar way, we show that $\Lambda^{-1}$ is continuous, then $\Lambda$ is a homeomorphism.
\end{proof} 	
	
	For such a semigroup action $(X,P,\theta)$, one can see that the   orbit $ [x]_{\theta} $ =
	$    \{\theta_{n}^{-1}(\theta_{m}(x)) | \;m,n \in P\}$ for $x\in X$. Thus we have the following lemma.

	\begin{lemma}
		Two semigroup actions  by homeomorphisms, $(X,P,\theta)$ and  $(Y,S,\rho)$, are  continuously  orbit equivalent   if  and only if there exist a homeomorphism $\varphi:\, X\rightarrow Y$,  continuous mappings
		$a_{1},b_{1}:\; P\times P\times X \rightarrow S$    and  $a_{2},b_{2}:\; S\times S\times Y \rightarrow P $ such that  $$\rho_{a_{1}(m,n,x)}(\varphi(x))=\rho_{b_{1}(m,n,x)}(\varphi(\theta_{n}^{-1}(\theta_{m}(x)))) \;\; \text{for} \; x\in X, m,n\in P, \eqno{(4.1)}$$
		$$ \theta_{a_{2}(s,t,y)}(\varphi^{-1}(y))= \theta_{b_{2}(s,t,y)}(\varphi^{-1}(\rho_{t}^{-1}(\rho_{s}(y))))\;\; \text{for} \; y\in Y, s,t \in S. \eqno{(4.2)}$$
		
	\end{lemma}
	
	Let $H$ be a countable group and $S$ be a subsemigroup of $H$ such that $S\cap S^{-1}=\{e\}$. One may define a left-invariant order $\leq$ on $H$ by saying that $x\leq y$ $ \Leftrightarrow  x^{-1}y\in S$. A pair $(H,S)$ is called a \emph{lattice-ordered group} if, for every $x$ and $y$ in $H$, the set $\{x,y\}$ admits a least upper bound $x\vee y$ and  a greatest lower bound $x\wedge y$.
	
	For each $g\in H$, we have $(g\wedge e)\leq e$, $(g\wedge e)\leq g$, $g\leq (g\vee e)$ and $e\leq (g\vee e) $. It follows that $(g\wedge e)^{-1}\in S$, $(g\wedge e)^{-1}g\in S$, $g^{-1}(g\vee e)\in S$ and $g\vee e\in S$. Thus $g=(g\wedge e)((g\wedge e)^{-1} g)\in S^{-1}S$ and $g=(g\vee e)((g\vee e)^{-1}g)\in SS^{-1}$. Thus, if $(H,S)$ is a lattice-ordered group, then $H=S^{-1}S=SS^{-1}$ (\cite[Proposition 8.1]{ER}).
	
	\begin{proposition}
		Let $(G,P)$ and $(H,S)$ be two lattice-ordered groups. For two semigroup actions by homeomorphisms $(X,P,\theta)$ and $(Y,S,\rho)$,  if two associated \'{e}tale groupoids  $\mathcal{G}(X,P,\theta)$ and   $\mathcal{G}(Y,S,\rho)$ are isomorphic, then $(X,P,\theta)\sim_{coe}(Y,S,\rho)$.
	\end{proposition}	
	\begin{proof}
		Assume that $\Lambda :\mathcal{G}(X,P,\theta) \rightarrow\mathcal{G}(Y,S,\rho)$ is an  isomorphism. Let $\varphi$ be the restriction of $\Lambda$ to $X$, and let $a(x,g,y)=c_{\rho}\Lambda(x,g,y)$, $b(u,h,v)=c_{\theta}\Lambda^{-1}(u,h,v)$, where $c_{\theta}$ and $c_{\rho}$ are the canonical cocycles on $\mathcal{G}(X,P,\theta)$ and $\mathcal{G}(Y,S,\rho)$. Then $\varphi:\, X\rightarrow Y$ is a homeomorphism, $\Lambda(x,g,y)=(\varphi(x),a(x,g,y),$ $\varphi(y))$ and $\Lambda^{-1}(u,h,v)=(\varphi^{-1}(u),b(u,h,v),$ $\varphi^{-1}(v))$.

Remark that for $x\in X, m,n\in P$, we have $\gamma=(x,mn^{-1},\theta_{n}^{-1}(\theta_{m}(x)))\in \mathcal{G}(X,P,\theta) $ and $\Lambda(\gamma)=(\varphi(x),a(\gamma),\varphi( \theta_{n}^{-1}(\theta_{m}(x)) ))\in \mathcal{G} (Y,S,\rho)$. Define two maps  $a_{1},b_{1}:  P\times P\times X$ $ \rightarrow S$  by
		$$a_{1}(m,n,x)=a(x,mn^{-1},\theta_n^{-1}(\theta_m(x)))\vee e$$ and
		$$b_{1}(m,n,x)= a(x,mn^{-1},\theta_n^{-1}(\theta_m(x)))^{-1} a_{1}(m,n,x)$$
		for $m,n\in P$ and $x\in X$. From the remark before this proposition, $a_1$ and $b_1$ are well-defined and $a(x,mn^{-1},\theta_n^{-1}(\theta_m(x)))=a_{1}(m,n,x)b_{1}(m,n,x)^{-1}$. It follows from the map $\Lambda$ that  $\rho_{a_{1}(m,n,x)}(\varphi(x))=  \rho_{b_{1}(m,n,x)}(\varphi(\theta_{n}^{-1}(\theta_{m}(x)))) $ for $x\in X$, $m,n\in P$.
		
		To see that     $a_{1},b_{1}$  are continuous,	suppose $(m_{i},n_{i},x_{i})\rightarrow (m,n,x) \in P\times P\times X$. Then $m_{i}=m, n_{i}=n$ for large $i$, so  we can assume that $m_{i}=m, n_{i}=n$ for all $i$. Denote by $y_{i}=\theta_{n}^{-1}(\theta_{m}(x_{i}))$ for each $i$ and $y=\theta_{n}^{-1}(\theta_{m}(x))$. Then $y_i\rightarrow y$. For an open subset $U\subseteq X$ with $x\in U$,  let $V= \theta_{n}^{-1}(\theta_{m}(U))$. Then $A=\Sigma(U,m,n,V)$ is an open bisection 	containing $ (x,mn^{-1},y)$, and      $(x_{i},mn^{-1},y_{i})\in A $ for large enough $i$, which implies   $(x_{i},mn^{-1},y_{i})\rightarrow (x,mn^{-1},y)$ in $ \mathcal{G}(X,P,\theta) $. Since  $a$ is continuous, we can assume that $a(x,mn^{-1},y)=a(\gamma)$ for each $ \gamma\in A $. Then $a_{1}(m_{i},n_{i},x_{i})=a_{1}(m,n,x) $ and $b_{1}(m_{i},n_{i},x_{i})=b_{1}(m,n,x) $ for larger $i$. Thus $a_{1} ,b_{1}$  are continuous.    Similarly,   we can construct  continuous maps $a_{2},b_{2}:\; S\times S\times Y \rightarrow P $ satisfying (4.2). It follows from Lemma 4.3 that $(X,P,\theta)\sim_{coe}(Y,S,\rho)$.
	\end{proof}

	Recall that a group action $ (X,G,\alpha)$ is  said to be \emph{topologically free} if for every $e\neq g\in G$, $\{x\in X:\, \alpha_g(x)\neq x\}$ is dense in $X$.
	By  definitions, one can  easily check that a semigroup   action by homeomorphisms,  $(X,P,\theta)$, is essentially free  if and only if the associated group action $(X,G,\tilde{\theta})$ is topologically free.
	By Theorem 3.9, Proposition 3.10, Proposition 4.2, Proposition 4.4 and \cite[Theorem 1.2]{Li1}, we have the following result.
	
	\begin{theorem}
		Let $(X,P,\theta)$ and  $(Y,S,\rho)$ be two essentially free semigroup actions by homeomorphisms. If
		$(X,P,\theta)\sim_{coe}(Y,S,\rho)$, then $ (X,G,\tilde{\theta})  \sim_{coe}(Y,H,\tilde{\rho})$  in Li's sense $ ( $\cite{Li1}$ ) $. Moreover, if $X$ and $Y$ are totally disconnected or $(G,P)$ and $(H,T)$ are two lattice-ordered groups, then the converse holds.
\end{theorem}

	\subsection*{Acknowledgements}
	This work is supported by the NSF of China (Grant No. 11771379, 11971419, 11271224).

\end{document}